\documentclass[preprint]{imsart}
\pdfoutput=1
\RequirePackage[OT1]{fontenc}
\usepackage{amsthm,amsmath,amssymb,natbib,bm,enumitem}
\RequirePackage[colorlinks,citecolor=blue,urlcolor=blue]{hyperref}

\startlocaldefs
\numberwithin{equation}{section}
\theoremstyle{plain}
\newtheorem{thm}{Theorem}[section]
\newtheorem{lemma}{Lemma}[section]
\theoremstyle{definition}
\newtheorem{definition}{Definition}[section]
\theoremstyle{remark}
\newtheorem{remark}{Remark}[section]

\endlocaldefs

\def\citeapos#1{\citeauthor{#1}'s (\citeyear{#1})}

\allowdisplaybreaks

\begin{document}

\begin{frontmatter}
\title{A sharp boundary for SURE-based admissibility for the Normal means problem under unknown scale}
\runtitle{SURE-based admissibility}

\begin{aug}
\author{\fnms{Yuzo} \snm{Maruyama}\thanksref{t1,m1}
\ead[label=e1]
{maruyama@csis.u-tokyo.ac.jp}}
\and
\author{\fnms{William, E.} \snm{Strawderman}
\thanksref{t2,m2}
\ead[label=e2]{straw@stat.rutgers.edu}}

\thankstext{t1}{This work was partially supported by KAKENHI \#25330035, \#16K00040.}
\thankstext{t2}{This work was partially supported by grants from the Simons Foundation (\#209035 and \#418098 to William Strawderman).}
\address{University of Tokyo\thanksmark{m1} and Rutgers University\thanksmark{m2} \\
\printead{e1,e2}}
\runauthor{Y. Maruyama and W. Strawderman}

\end{aug}

\begin{abstract}
 We consider quasi-admissibility/inadmissibility of Stein-type shrinkage estimators
 of the mean of a multivariate normal distribution with covariance matrix an
 unknown multiple of the identity.
 Quasi-admissibility/inadmissibility is defined in terms of non-existence/existence
 of a solution to a differential inequality based on Stein's unbiased risk estimate (SURE).
 We find a sharp boundary between quasi-admissible and quasi-inadmissible estimators
 related to the optimal James-Stein estimator.
 We also find a class of priors related to the Strawderman class in the known 
 variance case where the boundary between quasi-admissibility and quasi-inadmissibility
 corresponds to the boundary between admissibility and inadmissibility in the known
 variance case.
 Additionally,
 we also briefly consider generalization to the case of general spherically symmetric distributions
 with a residual vector.
\end{abstract}


\end{frontmatter}
\section{Introduction}
\label{sec:intro}
Let
\begin{equation}\label{eq:sampling_densities}
 X\sim N_p(\theta,\sigma^2I_p),\ S\sim \sigma^2\chi^2_n,
\end{equation}
where $X$ and $S$ are independent and $\theta$ and $\sigma^2$ are both unknown,
and where
\begin{equation}\label{eq:dim}
 p\geq 3, \quad n\geq 3.
\end{equation}
Consider the problem of estimating the mean vector $\theta$ 
under the loss function 
\begin{equation}\label{eq:scaled_quadratic_loss}
 L(\{\theta,\sigma^2\},d)=\|d-\theta\|^2/\sigma^2.
\end{equation}
We study the question of admissibility/inadmissibility of shrinkage-type estimators
of the form
\begin{equation}\label{eq:delta_phi}
 \delta_\phi(X,S)=\left(1-\phi(W)/W\right)X, 
\end{equation}
where $ W=\|X\|^2/S$.
We do so by examining the existence of solutions to a differential inequality
which arises from an unbiased estimate of the difference in risk
between $\delta_\phi$ and $\delta_{\phi+g}$. 
Hence we are more properly studying what may be termed quasi-admissibility
and quasi-inadmissibility of such estimators.
Quasi-inadmissibility implies inadmissibility under conditions of risk finiteness,
while quasi-admissibility is relatively weaker.

Stein in his unpublished lecture notes, \cite{Brown-1988}, \cite{Bock-1988}, \cite{Rukhin-1995} and
\cite{Brown-Zhao-2009} among others have studied the admissibility
question from this point of view (without necessarily using the term quasi-admissibility)
under known $\sigma^2$.
Of course, \cite{Brown-1971} has largely settled the admissibility/inadmissibility question
when $\sigma^2$ is known.

Our efforts focus generally on finding a boundary between quasi-admissibility and
quasi-inadmissibility for shrinkage estimators of the form \eqref{eq:delta_phi}.
(See Theorem \ref{thm:main}) We also apply the result to a class of generalized Bayes
estimators related to the class of \cite{Strawderman-1971} priors for the known variance
problem and establish a boundary on the tail behavior which also separates
quasi-admissibility from quasi-inadmissibility.

While minimaxity of shrinkage estimators in the unknown scale case has been extensively
studied by many authors, relatively little is known about admissibility in this case.
\cite{Strawderman-1973} and \cite{Zinodiny-etal-2011} 
gave a class of proper Bayes minimax and hence 
admissible estimators under unknown $\sigma^2$. 
Note that proper Bayes estimators by \cite{Strawderman-1973} and \cite{Zinodiny-etal-2011} 
are not of the form given by \eqref{eq:delta_phi} whereas 
generalized Bayes estimators by \cite{Maruyama-2003a},
\cite{Maru-Straw-2005} and \cite{Maruyama-Strawderman-2009} are of this form. 

While our results on quasi-admissibility do not settle the admissibility issue,
it seems likely to us that generalized Bayes estimators satisfying our conditions
for quasi-admissibility are admissible, perhaps under mild additional conditions.
We are decidedly not claiming that such a result would be easily established!
Certainly those found to be quasi-inadmissible are in fact inadmissible under conditions
of finiteness of risk.

An unbiased estimator of of the risk, $R(\{\theta,\sigma^2\},\delta_\phi)$,
for an estimator of the form
\eqref{eq:delta_phi} is given by
\begin{equation}\label{eq:unbiased_estimator_1}
 p+(n+2)D_\phi(W)
\end{equation}
where
\begin{equation}\label{eq:unbiased_estimator_D_phi}
D_\phi(w)=\frac{\{\phi(w)-2c_{p,n}\}\phi(w)}{w}-d_n \phi'(w)\left\{1+\phi(w)\right\},
\end{equation}
with
\begin{equation}\label{eq:unbiased_estimator_c_d}
 c_{p,n}=(p-2)/(n+2)\text{ and }d_n=4/(n+2).
\end{equation}
This result follows from 
\citeapos{Stein-1981} identity and well known identities for chi-square distributions
(see e.g.~\cite{Efron-Morris-1976}).
We may refer to \eqref{eq:unbiased_estimator_1} as a SURE estimate of risk and
to \eqref{eq:n+2D} below as a SURE estimate of difference in risk.
A sufficient condition for its validity is
that $\phi$ be absolutely continuous 
and that each term of $E[D_\phi(W)]$ be finite. 
Let $\varPhi$ be a family of functions $\phi$, satisfying these sufficient conditions,
\begin{equation}\label{conditions_varphi}
 \varPhi =\left\{\phi: E[D_\phi(W)]<\infty, \text{ absolute continuous}\right\}.
\end{equation}
 
If $\delta_{\phi+g}$ is of the form \eqref{eq:delta_phi} with $\phi(w)$ replaced by
$\phi(w)+g(w)$, an unbiased estimator of the difference in risk between $\delta_\phi$
and $\delta_{\phi+g}$,
\begin{equation*}
 R(\{\theta,\sigma^2\},\delta_\phi)-R(\{\theta,\sigma^2\},\delta_{\phi+g})
\end{equation*}
is given by 
\begin{equation}\label{eq:n+2D}
\begin{split}
 (n+2)\Delta(w;\phi,g)&=(n+2)\{D_\phi(w)-D_{\phi+g}(w)\} \\
&=(n+2)g(w)\{\Delta_1(w;\phi)+\Delta_2(w;\phi,g)\}
\end{split}
\end{equation}
where
\begin{align}\label{Delta_1}
\Delta_1(w;\phi)=2\frac{c_{p,n}-\phi(w)}{w}+d_n\phi'(w)
\end{align}
and
\begin{align}\label{Delta_2}
\Delta_2(w;\phi,g)=\frac{-g(w)}{w}+d_ng'(w)+d_n \frac{g'(w)}{g(w)}\{1+\phi(w)\}. 
\end{align}
One may find an estimator dominating $\delta_\phi$
by finding a non-zero solution $g(\cdot)\in\varPhi$ to the 
differential inequality $\Delta(w;\phi,g)\geq 0$, where $\Delta(w;\phi,g)$ is given by \eqref{eq:n+2D},
providing the resulting estimator has finite risk.
Here is the definition of quasi-admissibility and quasi-inadmissibility used in this paper:
\begin{definition}\label{def:quasi}
 \begin{enumerate}
  \item An estimator $\delta_\phi$ of the form \eqref{eq:delta_phi} is said to be
quasi-admissible if any solution $g(w)\in \varPhi$ of the inequality
$\Delta(w;\phi,g)\geq 0$ satisfies $g(w)\equiv 0$,
\item $\delta_\phi$ is said to be
quasi-inadmissible if there exists a solution, $g(w)\in \varPhi$, which
is non-vanishing on some open interval, to the differential inequality $\Delta(w;\phi,g)\geq 0$.
 \end{enumerate}
\end{definition}
For technical reasons we will restrict the class of $\phi(\cdot)$ to the subclass
$\varPhi_{A}$ of $\varPhi$, defined as follows,
\begin{align}
 \varPhi_{A}=\left\{\phi\in\Phi, \text{and }\phi \text{ satisfies \ref{AA1}, \ref{AA2}, \ref{AA3}, and \ref{AA4} below}\right\},
\end{align}
\begin{enumerate}[label= A\arabic*]
\item \label{AA1} $\phi(0)=0$ and $\phi(w)\geq 0$ for any $w\geq 0$,
\item \label{AA2} $\phi(w)$ has at most finitely many local extrema,
\item \label{AA3} $\phi'(w)$ has only finitely many discontinuities and $\phi'(w)$ is continuous
from the right at $0$.
\item \label{AA4} $\liminf_{w\to\infty} w\phi'(w)/\phi(w)\geq 0$ and
       $\limsup_{w\to\infty} w\phi'(w)/\phi(w)\leq 1$.
\end{enumerate}

Note that James-Stein-type estimators $(1-a/W)X$ with $\phi(w)\equiv a$ do not satisfy Assumption \ref{AA1}.
However such estimators are inadmissible 
and are dominated by the positive part version $(1-a/W)_+ X$ for which $\phi_+(w)=\min(w,a)$.
$\phi_+(w)$ does in fact satisfy Assumption \ref{AA1}. The positive part modification of any $\delta_\phi$
for which $\lim_{w\to 0}\phi(w)>0$ will similarly satisfy $\phi_+(0)=0$.
Assumption \ref{AA2} assumes that $\phi(w)$ does not oscillate excessively and that
$\lim_{w\to\infty}\phi(w)$ exists. Assumption \ref{AA3} is
used in controlling the local behavior of $\phi$ and of $\phi'$.
Assumptions \ref{AA1}--\ref{AA4} are
satisfied by linear estimators of the form $\delta(X)=\alpha X$
for $0\leq \alpha \leq 1$ and for which $\phi(w)=(1-\alpha) w$.
These estimators are unique proper Bayes and admissible in the normal case for $0\leq\alpha<1$.
As far as we know, Assumptions \ref{AA1}--\ref{AA4} cover the positive part version
of all minimax estimators in the literature.
We emphasize that while we address quasi-admissibility and inadmissibility only for
$\delta_\phi$ for $\phi\in\Phi_A\subset\Phi$, we allow competitive estimators
of the form $\delta_{\phi + g}$ for $g\in\Phi$.

In Section \ref{sec:main} we will show the following result, which establishes
\begin{equation*}
  \phi(w)= \frac{p-2}{n+2}- \frac{\beta_\star}{\log w}
\end{equation*}
as the asymptotic boundary between quasi-admissibility and quasi-inadmissibility
where
\begin{equation}\label{beta_star_intro}
 \beta_\star=\frac{d_n(1+c_{p,n})}{2}=\frac{2(p+n)}{(n+2)^2}.
\end{equation}
\begin{description}
	 \item[Quasi-admissibility:] 
If $\phi\in\varPhi_A$ and there exists $w_*$ and $b<1$
such that
\begin{align}\label{eq:DD1.intro}
 \phi(w)\geq \frac{p-2}{n+2}- b\frac{\beta_\star}{\log w}, \ \forall w\geq w_*,
\end{align}
		    then $\delta_\phi$ is quasi-admissible.
\item[Quasi-inadmissibility:]
If $\phi\in\varPhi_A$ and there exists $w_*$ and $b>1$
such that
\begin{align}\label{eq:DD2.intro}
 \phi(w)\leq \frac{p-2}{n+2}-b\frac{\beta_\star}{\log w}, \ \forall w\geq w_*,
\end{align}
then $\delta_\phi$ is quasi-inadmissible (and hence inadmissible).
\end{description}
In Section \ref{sec:Bayes}, we find a generalized Bayes estimator with
asymptotic behavior
\begin{equation*}
\lim_{w\to\infty} \log w\left(\frac{p-2}{n+2}-\phi(w)\right)=b\beta_\star,
\end{equation*}
for all $b>0$. The corresponding generalized prior is given by
\begin{equation*}
\frac{1}{\sigma^2}\times\frac{1}{\sigma^p}G(\|\theta\|/\sigma)
\end{equation*}
with
\begin{equation*}
 G(\|\mu\|)=
\int_0^1 \left(\frac{\lambda}{1-\lambda}\right)^{p/2}
\exp\left(-\frac{\lambda}{1-\lambda}\frac{\|\mu\|^2}{2}\right)\lambda^{-2}
\left(\log\frac{1}{\lambda}\right)^b
  d\lambda.
\end{equation*}
Hence, $b<1$ and $b>1$ imply quasi-admissibility and quasi-inadmissibility, respectively,
of the associated generalized Bayes estimators.
Interestingly, 
the boundary $b=1$ also appears in the known $\sigma^2$ case when estimating $\mu$
with $Z\sim N_p(\mu,I_p)$. By using \citeapos{Brown-1971}
sufficient condition, the generalized Bayes estimator with respect to $G(\|\mu\|)$ above
is admissible (resp.~inadmissible) when $b\leq 1$ (resp.~$b>1$).
This nice correspondence leads naturally to the conjecture:
a quasi-admissible generalized Bayes estimator satisfying \eqref{eq:DD1.intro} is admissible. 

An extension to the general class of spherically symmetric distributions is briefly
considered in Section \ref{sec:ssd}. We give some concluding remarks in Section \ref{sec:cr}.
Some technical proofs are given in \ref{sec:AP}ppendix.

 \section{Quasi-admissibility}
\label{sec:main}
The main result of this paper, Theorem \ref{thm:main}, gives sufficient conditions for quasi-admissibility
and quasi-inadmissibility for estimators $\delta_\phi$ of the form \eqref{eq:delta_phi},
for $\phi\in\varPhi_A$.
In preparation, we first give several lemmas.
Recall that the unbiased estimator of the difference in risk between $\delta_\phi$
and $\delta_{\phi+g}$ is given by 
\begin{equation}\label{eq:n+2D_sec2}
\begin{split}
 (n+2)\Delta(w;\phi,g)&=(n+2)\{D_\phi(w)-D_{\phi+g}(w)\} \\
&=(n+2)g(w)\{\Delta_1(w;\phi)+\Delta_2(w;\phi,g)\}
\end{split}
\end{equation}
where
\begin{align}\label{Delta_1_sec2}
\Delta_1(w;\phi)=2\frac{c_{p,n}-\phi(w)}{w}+d_n\phi'(w)
\end{align}
and
\begin{align}\label{Delta_2_sec2}
\Delta_2(w;\phi,g)=\frac{-g(w)}{w}+d_ng'(w)+d_n \frac{g'(w)}{g(w)}\{1+\phi(w)\},
\end{align}
and where $c_{p,n}=(p-2)/(n+2)$ and $d_n=4/(n+2)$.
Note that $\Delta_2(w;\phi,g)$ is well-defined for $w$ such that $g(w)\neq 0$,
but $\Delta(w;\phi,g)$ is well-defined even when $g(w)=0$.

The first lemma gives necessary conditions on $g(w)$ for $\Delta(w;\phi,g)$
to be nonnegative for all $w\geq 0$.
\begin{lemma}\label{lem:GG}
Suppose $\Delta(w;\phi,g)\geq 0$ for all $w\geq 0$ with $\phi\in\varPhi_A$ and $g\in\varPhi$.
Then
\begin{enumerate}
\item \label{BB1} $g(0)\geq 0$, 
\item \label{BB2} $ g(w)\geq 0 $ for all $w> 0$,
\item \label{BB3} Suppose $g(w_0)>0$. Then, for any $w\geq w_0$, $g(w)>0$.
\end{enumerate}
\end{lemma}
\begin{proof}
 Section \ref{sec.ap.1} in Appendix.
\end{proof}
Recall that finiteness of $E[\phi(W)^2/W]$ is a necessary condition for $\phi$ to be in $\varPhi$.
Lemma \ref{CC0} below provides a necessary condition for
$E[\phi(W)^2/W]$ to be finite and hence for $\phi$ to be in $\varPhi$.
It is needed in the proof of Lemma \ref{lem:CC}.
\begin{lemma}\label{CC0}
A necessary condition for $E[\phi(W)^2/W]$ to be finite for any $(\theta,\sigma^2)$ is that
\begin{equation}\label{eq:liminfphi}
 \liminf_{t\to\infty}|\phi(t)|^{d_n}/t=0.
\end{equation}
\end{lemma}
\begin{proof}
 Section \ref{sec.ap.2} in Appendix.
\end{proof}
 Let $\mathcal{G}\subset \varPhi$ be a class of nonnegative functions which satisfy \ref{BB1}, \ref{BB2} and
 \ref{BB3} of Lemma \ref{lem:GG} and Lemma \ref{CC0}.
 The following lemma is key in proving the main result.
Recall that Assumption \ref{AA2} assumes that $\phi(w)$ does not oscillate excessively and that
$\lim_{w\to\infty}\phi(w)$ exists.
In the following lemma, let $\phi_*=\lim_{w\to\infty} \phi(w) \in [0,\infty]$ and
\begin{equation}\label{beta_star_star}
 \beta_\star=\frac{d_n(1+c_{p,n})}{2}=\frac{2(p+n)}{(n+2)^2}.
\end{equation}
\begin{lemma}\label{lem:CC}
 Suppose $\phi\in\varPhi_A$.
 \begin{enumerate}
\item \label{CC1}
      Suppose $\phi_*<\infty$ and there exists $w_0$ and $b<1$ such that
\begin{align}\label{eq:CC1.lem}
 \phi(w)\geq \frac{p-2}{n+2}- b\frac{\beta_\star}{\log w}, \ \forall w\geq w_0.
\end{align}
\begin{enumerate}[label= (1.\alph*), ref= 1.\alph*]
 \item \label{CC1.1}For all $w\geq w_0$,
\begin{equation}\label{eq:CC1.1.lem}
 \Delta_1(w;\phi)-d_n\phi'(w)-\frac{2b\beta_\star}{w\log w}\leq 0.
\end{equation}
\item \label{CC1.2}For any $g\in\mathcal{G}$ except $g\equiv 0$,
there exists $w_*\in(w_0,\infty)$ such that       
\begin{equation}\label{eq:CC1.2.lem}
\Delta_2(w_*;\phi,g)+d_n\phi'(w_*)+\frac{2b\beta_\star}{w_*\log w_*}<0.
\end{equation}
\end{enumerate}      
\item \label{CC2}
      Suppose $\phi_*=\infty$.
\begin{enumerate}[label= (2.\alph*), ref= 2.\alph*]
 \item \label{CC2.1}
       There exists $w_0$ such that
\begin{equation}\label{eq:CC2.1.lem}
 \Delta_1(w;\phi)+d_n\phi'(w)\leq 0,\text{ for all }w\geq w_0.
\end{equation}
 \item  \label{CC2.2}
	For any $g\in\mathcal{G}$ except $g\equiv 0$,
there exists $w_*\in(w_0,\infty)$ such that
\begin{equation}\label{eq:CC2.2.lem}
\Delta_2(w_*;\phi,g)-d_n\phi'(w_*)<0.
\end{equation}
\end{enumerate}
 \item \label{CC3}
Suppose there exists $w_0$ and $b>1$ such that
\begin{align}\label{eq:CC3.lem}
 \phi(w)\leq \frac{p-2}{n+2}-b\frac{\beta_\star}{\log w}, \ \forall w\geq w_0.
\end{align}
\begin{enumerate}[label= (3.\alph*), ref= 3.\alph*]
 \item \label{CC3.1}
       There exists $w_1$ such that
       \begin{equation}\label{eq:CC3.1.lem}
 \Delta_1(w;\phi)-\frac{2b\beta_\star}{w\log w}\geq 0,\text{ for all }w\geq w_1.
       \end{equation}
 \item \label{CC3.2}
       Fix
\begin{equation}
 \nu=\min\left(1,\frac{2b\beta_\star-d_n(1+\phi_*)}{2d_n(3+\phi_*)}\right).
\end{equation}
       Let $k(w)\in \mathcal{G}$ be any non-decreasing continuous function with
$k(0)=0$, $w_\sharp=\sup\{w: k(w)=0\}$  and $k(\infty)=1$.
Then there exists $w_*$, independent of $k(w)$, such that
\begin{equation}\label{eq:CC3.2.lem}
 \Delta_2(w;\phi,k(w)\{\log (w+e)\}^{-1-\nu})
+\frac{2b\beta_\star}{w\log w} \geq 0
\end{equation}
       for all $w> \max(\max(w_*,w_1), w_\sharp)$ and  $e=\exp(1)$.
 \end{enumerate}
\end{enumerate}
\end{lemma}
\begin{proof}
Section \ref{sec.ap.3} in Appendix.
\end{proof}
Note, in part \ref{CC3}, $ \Delta_2(w;\phi,k(w)\{\log (w+e)\}^{-1-\nu})$ is well-defined
for $w>w_\sharp$ by the definition of $\Delta_2$ given by \eqref{Delta_2_sec2}.

\smallskip

The following result is the main result of this section.
\begin{thm}\label{thm:main}
 Suppose $\phi\in\varPhi_A$.
 \begin{enumerate}
  \item \label{DD1}
[quasi-admissibility] If there exists $w_*$ and $b<1$
such that
\begin{align}\label{eq:DD1}
 \phi(w)\geq \frac{p-2}{n+2}- b\frac{\beta_\star}{\log w}, \ \forall w\geq w_*,
\end{align}
then $\delta_\phi$ is quasi-admissible.
  \item \label{DD2}
	[quasi-inadmissibility]
If there exists $w_*$ and $b>1$
such that
\begin{align}\label{eq:DD2}
 \phi(w)\leq \frac{p-2}{n+2}-b\frac{\beta_\star}{\log w}, \ \forall w\geq w_*,
\end{align}
then $\delta_\phi$ is quasi-inadmissible (and hence inadmissible).
 \end{enumerate}
\end{thm}
\begin{proof}\mbox{}
 [Part \ref{DD1}]
 By Parts \ref{CC1} ($\phi_*<\infty$) and \ref{CC2} ($\phi_*=\infty$) of Lemma \ref{lem:CC},
there exists $w_*$ such that $ \Delta_1(w_\star;\phi)+\Delta_2(w_\star;\phi,g)<0$ 
for any $g\in\mathcal{G}$ except $g\equiv 0$. 
 Therefore any solution $g(w)\in\mathcal{G}$ of the differential inequality 
\begin{equation*}
 g(w)\left\{\Delta_1(w;\phi)+\Delta_2(w;\phi,g)\right\}\geq 0
\end{equation*}
must be identically equal to $0$,
or equivalently $\delta_\phi$ is quasi-admissible.

 [Part \ref{DD2}]
By \eqref{eq:DD2}, we have $\phi_*\leq (p-2)/(n+2)=c_{p,n}$ and hence
 \begin{equation}\label{futoushiki_2.1}
  d_n(1+\phi_*)\leq 2\beta_\star<2b\beta_\star
 \end{equation}
since $b>1$. 
As in Part \ref{CC3} of Lemma \ref{lem:CC}, let 
\begin{equation}
 \nu=\min\left(1,\frac{2b\beta_\star-d_n(1+\phi_*)}{2d_n(3+\phi_*)}\right).
\end{equation}
 Take any $k(w)$ with $ w_\sharp=\max(w_1, w_*)$ where $w_1$ and $w_1$ are both determined by
Part \ref{CC3} of Lemma \ref{lem:CC}.
Let  $g(w)=k(w)\{\log (w+e)\}^{-1-\nu}\in\mathcal{G}$.
 Then we have
\begin{align*}
 \Delta(w)= g(w)\left\{\Delta_1(w;\phi)+\Delta_2(w;\phi,g)\right\}
 \begin{cases}
 =0 & 0\leq w\leq w_\sharp \\
\geq 0 & w> w_\sharp,
 \end{cases}
\end{align*}
 where $ \Delta(w)=0$ for $0\leq w\leq w_\sharp$ since $g(w)=0$ and $ \Delta(w)\geq 0$ for $w> w_\sharp$
 since
\begin{equation*}
 \left( \Delta_1(w;\phi)-\frac{2b\beta_\star}{w\log w} \right)
  +\left(\Delta_2(w;\phi,g)
+\frac{2b\beta_\star}{w\log w}\right)\geq 0
\end{equation*}
by Part \ref{CC3} of Lemma \ref{lem:CC}.
Hence $\delta_\phi$ is quasi-inadmissible.
\end{proof}

\begin{remark}
 Note that it is possible that an estimator which is quasi-admissible
according to the above definition may fail to be admissible for several reasons.
Here are some of them.
First, there may be an estimator that is not of the form  \eqref{eq:delta_phi} 
that dominates $\delta_\phi$. Second,
there may be an estimator of the form  \eqref{eq:delta_phi} with $g(w)\notin \varPhi$ 
that dominates $\delta_\phi$.
Third there may be an estimator that dominates $\delta_\phi$
but does not satisfy the differential inequality $\Delta(w;\phi,g)\geq 0$.
Hence quasi-admissibility is quite weak as an optimality criterion.

Quasi-inadmissibility, on the other hand, is more compelling in the sense that
if $\delta_\phi$ is quasi-inadmissible then it is inadmissible and dominated by $\delta_{\phi+g}$.
Note that requiring both $\phi$ and $g$ to be in $\varPhi$ implies that 
the risk of $\delta_{\phi+g}$ is finite.
\end{remark}

\subsection{General spherically Symmetric distributions}
\label{sec:ssd}
We may also study the more general canonical spherically symmetric setting where
$(X,U)$ has a spherically symmetric density of the form 
\begin{equation}\label{density_f.ssd}
 \sigma^{-p-n}f(\{\|x-\theta\|^2+\|u\|^2\}/\sigma^2).
\end{equation}
Here the $p$-dimensional vector $X$ has mean vector $\theta$, the $n$-dimensional ``residual''
vector $U$ has mean vector $0$ and $(X,S)$ is sufficient, where $S=\|U\|^2$.
The scale parameter, $\sigma^2$, is assumed unknown.
Consider the problem of estimating the mean vector $\theta$ 
under the loss function 
\begin{equation}\label{eq:scaled_quadratic_loss.ssd}
 L(\{\theta,\sigma^2\},d)=\|d-\theta\|^2/\sigma^2.
\end{equation}
The most important such setting is the Gaussian case 
\begin{equation}\label{eq:sampling_densities.ssd}
 X\sim N_p(\theta,\sigma^2I_p),\ S\sim \sigma^2\chi^2_n,
\end{equation}
which is studied in Section \ref{sec:main}, but there is considerable interest
in the case of heavier tailed distributions such as the multivariate-$t$.

In the general spherically symmetric case,  
\eqref{eq:unbiased_estimator_1} is not an unbiased estimate of risk but has been used as a substitute for such an estimator.
In particular, if $(X,S)$ has density \eqref{density_f.ssd} and $F(\cdot)$ is defined as
\begin{equation}
 F(t)=\frac{1}{2}\int_t^\infty f(v)dv.
\end{equation}
Then as essentially shown by several authors in various settings
(see e.g.~\cite{Kubokawa-Srivastava-2001} and \cite{Fourdrinier-Strawderman-2014})
\begin{align*}
&R(\{\theta,\sigma^2\},\delta_{\phi}) \\ &=p+(n+2)\int_{\mathbb{R}^{p+n}}D_\phi(w)
 \frac{F(\{\|x-\theta\|^2+\|u\|^2\}/\sigma^2)}{\sigma^{p+n}}dx du,
\end{align*}
where $ D_\phi(w)$ is given in \eqref{eq:unbiased_estimator_D_phi}.
Hence, in this setting,
 \begin{align*}
&  R(\{\theta,\sigma^2\},\delta_\phi)-R(\{\theta,\sigma^2\},\delta_{\phi+g}) \\
  &=(n+2)\int_{\mathbb{R}^{p+n}} g(w)\{\Delta_1(w;\phi)+\Delta_2(w;\phi,g)\}
  \frac{F(\{\|x-\theta\|^2+\|u\|^2\}/\sigma^2)}{\sigma^{p+n}}dx du
 \end{align*}
 where $w=\|x\|^2/\|u\|^2$.
Thus, study of existence of solutions to $\Delta(w)\geq 0$ is relevant in the general spherically
symmetric case as well as in the Gaussian case,
and defining quasi-admissibility/inadmissibility as in Definition \ref{def:quasi} implies
that Theorem \ref{thm:main} remains valid in this more general setting.

\section{Generalized Bayes estimators in the Normal case}
\label{sec:Bayes}
\subsection{Known variance case}
Let $Z\sim N_p(\mu,I_p)$. Consider estimation of $\mu$ under quadratic loss $\|\hat{\mu}-\mu\|^2$.
The MLE, $Z$ itself, is inadmissible for $p\geq 3$ as shown in \cite{Stein-1956}.
\cite{Brown-1971} showed that admissible estimators should be proper Bayes or generalized Bayes
estimators with respect to an improper prior and gave 
a sufficient condition for generalized Bayes estimators to be admissible or inadmissible.

Let the prior be of the form
\begin{equation}\label{general_spherical_prior_known}
 \pi(\mu)=G(\|\mu\|;a,L)
\end{equation}
where
\begin{equation}\label{GGG}
 G(\|\mu\|;a,L)=
\int_0^1 \left\{\frac{\lambda}{1-\lambda}\right\}^{p/2}
  \exp\left(-\frac{\lambda}{1-\lambda}\frac{\|\mu\|^2}{2}\right)
  \lambda^a L(1/\lambda)  d\lambda
\end{equation}
where $p/2+a+1>0$. We assume the following on $L$:$[1,\infty)$ $\to$ $[0,\infty)$
\begin{enumerate}[label= L\arabic*]
 \item \label{L.A.3}
       $L(y)$ is slowly varying at infinity, that is, for all $c>0$, 
\begin{equation*}
 \lim_{y\to\infty}L(cy)/L(y)=1.
\end{equation*}
 \item \label{L.A.1}
       $L(y)$ is ultimately monotone,
 \item \label{L.A.2}
       $L(y)$ is differentiable with ultimately monotone derivative $ L'(y)$,
\end{enumerate}
By Proposition 1.7 (11) of \cite{geluk-dehaan-1987}, Assumptions \ref{L.A.3} and \ref{L.A.2} implies
\begin{equation*}
 \lim_{y\to \infty}y\frac{L'(y)}{L(y)}=0.
\end{equation*}
Under the prior given by \eqref{GGG}, the marginal density is
\begin{align*}
 m(\|z\|;a,L)&=\int_{\mathbb{R}^p}\frac{1}{(2\pi)^{p/2}}
 \exp\left(-\frac{\|z-\mu\|^2}{2}\right)G(\|\mu\|;a,L)d\mu \\
&=\int_0^1
 \exp\left(-\frac{\lambda\|z\|^2}{2}\right)\lambda^{p/2+a}L(1/\lambda)d\lambda \\
&=\int_0^\infty
 \exp\left(-\frac{\lambda\|z\|^2}{2}\right)f(\lambda;a,L)d\lambda,
\end{align*}
where $f(\lambda;a,L)=\lambda^{p/2+a}L(1/\lambda)I_{(0,1)}(\lambda)$.
Note that $f(\lambda;a,L)$ is ultimately monotone as a function of $1/\lambda$ since
\begin{enumerate}
 \item When $p/2+a=0$, $L(1/\lambda)$ itself is ultimately monotone. 
 \item When $p/2+a\neq 0$, $ \lim_{\lambda\to 0}\lambda f'(\lambda)/f(\lambda)=p/2+a\neq 0$.
\end{enumerate}
Since $f(\lambda;a,L)$ is ultimately monotone and since
$m(\|z\|;a,L)$ is the Laplace transform of $f$, 
a Tauberian Theorem 
(See e.g.~\cite{Feller-1971} Theorem 13.5.4) implies that
\begin{equation}\label{TaubTaub}
 m(\|z\|;a,L)  \approx \Gamma(p/2+a+1)\left(2/\|z\|^{2}\right)^{p/2+a+1}L(\|z\|^2)
\end{equation}
as $\|z\|\to\infty$.
As shown in Appendix \ref{sec:ap.bound}, 
$\|z\|\|\nabla \log m(\|z\|;a,L)\|$ is bounded.
By Theorem 6.4.2 of \cite{Brown-1971},
divergence (convergence) of the integral
\begin{equation}\label{BrownBrown}
 \int_1^\infty \frac{dr}{r^{p-1}m(r;a,L)}
\end{equation}
corresponds to admissibility (inadmissibility) of a generalized Bayes estimator
with bounded $\|z\|\|\nabla \log m(\|z\|;a,L)\|$.
Hence, by \eqref{TaubTaub} and \eqref{BrownBrown}, 
divergence (convergence) of the integral
\begin{equation}\label{BrownBrown}
 \int_1^\infty \frac{r^{2a+3}}{L(r^2)}dr
\end{equation}
corresponds to admissibility (inadmissibility).
It is clear that $a>-2$ and $a<-2$ imply admissibility and inadmissibility, respectively.

When $a=-2$, the fact that
\begin{equation}
 \int_1^\infty \frac{dr}{r\{\log r\}^b}
  \begin{cases}
   = \infty & b\leq 1 \\ <\infty & b>1
  \end{cases}
\end{equation}
is helpful to determine the boundary.
Since
\begin{equation}\label{eq:log_b}
\int_1^\infty \frac{r^{2a+3}}{L(r^2)}dr 
=
\frac{1}{2^b}\int_1^\infty \frac{1}{r\{\log r\}^b}\frac{\{\log r^2\}^b}{L(r^2)}dr,
\end{equation}
we have a following result on admissibility and inadmissibility of the
(generalized) Bayes estimator with respect to $G(\|\mu\|;a,L)$.
\begin{thm}[Admissibility]\label{thm:ad}\mbox{}
\begin{enumerate}
 \item \label{thm:ad.p1}
       Suppose $a>-2$. The generalized Bayes estimator is inadmissible.
 \item \label{thm:ad.p2}
       Suppose $a=-2$ and $\log(y)/L(y)$ is ultimately monotone non-decreasing.
       The generalized Bayes estimator is admissible.
\end{enumerate}
\end{thm}
\begin{thm}[Indmissibility]\label{thm:inad}\mbox{}
\begin{enumerate}
 \item \label{thm:inad.p1}Suppose $a=-2$ and $\{\log(y)\}^b/L(y)$ for $b>1$ is ultimately monotone non-increasing.
The generalized Bayes estimator is inadmissible.
 \item \label{thm:inad.p2}
       Suppose $a<-2$. The generalized Bayes estimator is inadmissible.
\end{enumerate}
\end{thm}
\begin{remark}[A boundary esimator for the known variance case]\label{rem:b.1}\mbox{}
 Consider the particular choice
\begin{equation}\label{eq:Lb}
a=-2\text{ and } L(1/\lambda)=\left(\log\frac{1}{\lambda}\right)^b \text{ for }b>0.
\end{equation}
Then the prior is given by
\begin{equation*}
\int_0^1 \left\{\frac{\lambda}{1-\lambda}\right\}^{p/2}
  \exp\left(-\frac{\lambda}{1-\lambda}\frac{\|\mu\|^2}{2}\right)
  \lambda^{-2} \left(\log\frac{1}{\lambda}\right)^b  d\lambda. 
\end{equation*}
By following \cite{Strawderman-1971}, the corresponding generalized Bayes estimator 
is $(1-\psi_{-2,b}(\|Z\|^2)/\|Z\|^2)Z$ where
\begin{equation*}
 \psi_{-2,b}(v)=v\frac
  {\int_0^1\lambda^{p/2-1}\{\log(1/\lambda)\}^b\exp(-v\lambda/2)d\lambda} 
  {\int_0^1\lambda^{p/2-2}\{\log(1/\lambda)\}^b\exp(-v\lambda/2)d\lambda},
\end{equation*}
As shown in Appendix \ref{sec.ap.5}, we have
\begin{equation}\label{known_variance_phi_asymptotic}
 \lim_{v\to\infty}\left(\log v\right)\left\{p-2-\psi_{-2,b}(v)\right\}=2b.
\end{equation}
 Hence
 by Part \ref{thm:ad.p2} of Theorem \ref{thm:ad} and 
Part \ref{thm:inad.p1} of Theorem \ref{thm:inad}, 
the generalized Bayes estimator with asymptotic behavior
 \begin{equation*}
  \left(1-\left\{p-2 - \frac{b}{\log \|Z\|}\right\}\frac{1}{\|Z\|^2}\right)Z
 \end{equation*}
is admissible and inadmissible for $ b\leq 1$ and $ b>1$. 
Thus the estimator
 \begin{equation*}
  \left(1-\left\{p-2 - \frac{1}{\log \|Z\|}\right\}\frac{1}{\|Z\|^2}\right)Z
 \end{equation*}
 is a boundary estimator. See also Corollary 6.3.2 of \cite{Brown-1971} and
 Theorem 6.1.1 of \cite{Strawderman-Cohen-1971} for related discussions, but where
 the $b/\log \|z\|$ term is not included.
\end{remark}

\subsection{Unknown variance case}
Let $X$ and $S$ be given by \eqref{eq:sampling_densities} and let the prior be of the form
\begin{equation}\label{general_spherical_prior}
 \pi(\theta,\sigma^2)=\frac{1}{\sigma^2}\pi(\theta|\sigma^2)=\frac{1}{\sigma^2} \times \frac{1}{\sigma^p}G(\|\theta\|/\sigma)
\end{equation}
where $G$ is given by \eqref{GGG} and $1/\sigma^2$ is a standard non-informative prior for $\sigma^2$.

The following two theorems relate quasi-admissibility/inadmissibility in the unknown
variance case to admissibility/inadmissibility in the known variance case as given in
Theorems \ref{thm:ad} and \ref{thm:inad}.
\begin{thm}[Quasi-admissibility]\label{thm:q.ad}\mbox{}
 \begin{enumerate}
 \item Suppose $a>-2$. The generalized Bayes estimator is quasi-admissible.
 \item Suppose $a=-2$ and $\{\log(y)\}^b/L(y)$ for $b<1$ is monotone non-decreasing.
       The generalized Bayes estimator is quasi-admissible.
 \end{enumerate}
\end{thm}
\begin{thm}[Quasi-indmissibility]\label{thm:q.inad}\mbox{}
\begin{enumerate}
 \item Suppose $a=-2$ and $\{\log(y)\}^b/L(y)$ for $b>1$ is monotone non-increasing.
       The generalized Bayes estimator is quasi-inadmissible.
 \item Suppose $a<-2$. The generalized Bayes estimator is quasi-inadmissible.
\end{enumerate}
\end{thm}
\begin{proof}[Proof of Theorems \ref{thm:q.ad} and \ref{thm:q.inad}]
By following \cite{Maru-Straw-2005} and \cite{Maruyama-Strawderman-2009}, 
the generalized Bayes estimator under the prior given by \eqref{general_spherical_prior}
is $\delta_\phi$ with 
\begin{equation*}
 \phi_{a,L}(w)
  =w\frac{\int_0^1\lambda^{p/2+a+1}L(1/\lambda)(1+w\lambda)^{-(p+n)/2-1}d\lambda}
  {\int_0^1\lambda^{p/2+a}L(1/\lambda)(1+w\lambda)^{-(p+n)/2-1}d\lambda}.
\end{equation*}
By a change of variables ($t=w\lambda$), we have
\begin{equation*}
  \phi_{a,L}(w)
=\frac{\int_0^w t^{p/2+a+1}L(w/t)(1+t)^{-(p+n)/2-1}dt}
  {\int_0^w\lambda^{p/2+a}L(w/t)(1+t)^{-(p+n)/2-1}dt}. 
\end{equation*}
By Assumption \ref{L.A.3} and the Lebesgue dominated convergence theorem,
\begin{equation}\label{phi.L.limit}
 \begin{split}
 \lim_w \phi_{a,L}(w)
&=\frac{\int_0^\infty t^{p/2+a+1}(1+t)^{-(p+n)/2-1}dt}
  {\int_0^\infty t^{p/2+a}(1+t)^{-(p+n)/2-1}dt} \\
&=\frac{p/2+a+1}{n/2-a-1}
\end{split}
\end{equation}
which is increasing in $a$ and is equal to $(p-2)/(n+2)$ when $a=-2$.
Hence, by Theorem \ref{thm:main},
$a>-2$ and $a<-2$ implies quasi-admissibility and quasi-inadmissibility, respectively.

When $a=-2$, take
\begin{equation}
 L(1/\lambda)=\left\{\log\frac{1}{\lambda}\right\}^{b}
\end{equation}
for $b>0$ and consider 
\begin{equation}\label{eq:phi-2b}
 \phi_{-2,b}(w)
  =w\frac{\int_0^1\lambda^{p/2-1}\{\log(1/\lambda)\}^b(1+w\lambda)^{-(p+n)/2-1}d\lambda}
  {\int_0^1\lambda^{p/2-2}\{\log(1/\lambda)\}^b(1+w\lambda)^{-(p+n)/2-1}d\lambda}.
\end{equation}
Then we have
\begin{equation}\label{lim.final.qa}
 \lim_{w\to\infty}(\log w)\left(\frac{p-2}{n+2}- \phi_{-2,b}(w)\right)=b\frac{2(p+n)}{(n+2)^2}=b\beta_\star
\end{equation}
where $\beta_\star$ is given by Theorem \ref{thm:main}.
See Section \ref{sec.ap.5} in Appendix for the derivation of \eqref{lim.final.qa}.
Further 
the inequality
\begin{equation}\label{final.ineq}
 \begin{split}
 \phi_{-2,b}(w)&=\frac{\int_0^1\lambda^{p/2-1}\{\log(1/\lambda)\}^b(1+w\lambda)^{-(p+n)/2-1}d\lambda}
  {\int_0^1\lambda^{p/2-2}\{\log(1/\lambda)\}^b(1+w\lambda)^{-(p+n)/2-1}d\lambda} \\
 &=\frac{\displaystyle \int_0^1\lambda \frac{\{\log(1/\lambda)\}^b}{L(1/\lambda)}
 \lambda^{p/2-2}L(1/\lambda)(1+w\lambda)^{-(p+n)/2-1}d\lambda}
 {\displaystyle \int_0^1 \frac{\{\log(1/\lambda)\}^b}{L(1/\lambda)}
 \lambda^{p/2-2}L(1/\lambda)(1+w\lambda)^{-(p+n)/2-1}d\lambda} \\
 & \leq (\geq)
 \frac{ \int_0^1\lambda 
 \lambda^{p/2-2}L(1/\lambda)(1+w\lambda)^{-(p+n)/2-1}d\lambda}
 {\int_0^1 
 \lambda^{p/2-2}L(1/\lambda)(1+w\lambda)^{-(p+n)/2-1}d\lambda} \\ &=\phi_{-2,L}(w)
\end{split}
\end{equation}
follows for $b>0$
when $ \{\log(y)\}^b/L(y)$ is monotone non-increasing (non-decreasing).
From \eqref{phi.L.limit}, \eqref{lim.final.qa}, \eqref{final.ineq} and Theorem \ref{thm:main},
 the two theorems follow.
\end{proof}
\begin{remark}[A boundary esimator for the unknown variance case]\label{rem:b.2}\mbox{}
 For the unknown variance case, Theorem \ref{thm:main} established the boundary
 between quasi-admissibility and quasi-inadmissibility for estimator
 of the form $(1-\phi(W)/W)X$ as
 \begin{equation}\label{eq:bound.unknown}
  \left(1-\left\{\frac{p-2}{n+2} - \frac{b\beta_*}{\log \|W\|}\right\}\frac{1}{\|W\|^2}\right)X
 \end{equation}
 with $b<1$ corresponding to quasi-admissibility and
 $b>1$ corresponding to quasi-inadmissibility.
The generalized prior
\begin{equation*}
 \pi(\theta,\sigma^2)=\frac{1}{\sigma^2} \times \frac{1}{\sigma^p}G(\|\theta\|/\sigma)
\end{equation*}
with $ G$ given by \eqref{GGG} where
\begin{equation}\label{eq:Lb.1}
a=-2\text{ and } L(1/\lambda)=\left(\log\frac{1}{\lambda}\right)^b \text{ for }b>0
\end{equation}
leads to a generalized Bayes estimator with $\phi$ given in \eqref{eq:phi-2b}.
As shown in Appendix, the asymptotic behavior of this $\phi$ is
\begin{equation}\label{lim.final.qa.1}
 \lim_{w\to\infty}(\log w)\left(\frac{p-2}{n+2}- \phi(w)\right)=b\frac{2(p+n)}{(n+2)^2}=b\beta_\star.
\end{equation}
Thus we see that the behavior of the generalized Bayes shrinkage function in the cases
of known and unknown scale for the related classes of priors
are in very close correspondence. Additionally admissibility/inadmissibility in the known scale
case corresponds exactly with quasi-admissibility/inadmissibility in the unknown scale case.
We conjecture, for this class of priors in the unknown scale case, that
quasi-admissibility/inadmissibility in fact corresponds to admissibility/inadmissibility.

\end{remark}

 \section{Concluding remark}
\label{sec:cr}
We have studied quasi-admissible and quasi-inadmissible Stein-type shrinkage estimators
in the problem of estimating the mean vector of a $p$-variate Normal distribution when
the covariance matrix is an unknown multiple of the identity. We have established
sharp boundary of the form
\begin{equation}
 \phi_\star(w)=\frac{p-2}{n+2}-\frac{\beta_\star}{\log w}
\end{equation}
where $\beta_\star=2(p+2)/(n+2)^2$.
Roughly, estimators with shrinkage function $\phi(w)$ ultimately less than $\phi_\star(w)$
are quasi-inadmissible, while those which ultimately shrink more are quasi-admissible.
We have also found generalized prior distributions of the form $(1/\sigma^2)\times (1/\sigma^p)G(\|\theta\|/\sigma)$
for which he resulting generalized Bayes estimators are asymptotically of the form
\begin{equation*}
 \left\{1-\left(\frac{p-2}{n+2}-\frac{b\beta_\star}{\log W}\right)\frac{1}{W}\right\}X
\end{equation*}
for any $b>0$, thus establishing a boundary behavior for this class of priors
between quasi-admissibility and quasi-inadmissibility.
We conjecture, for this class of priors, that
quasi-admissibility/inadmissibility in fact corresponds to admissibility/inadmissibility.

\appendix
\section{Proofs}
\label{sec:AP}
\subsection{Proof of Lemma \ref{lem:GG}}
\label{sec.ap.1}
Let 
\begin{align}
\Delta(w)=\Delta(w;\phi,g), \ \Delta_1(w)=\Delta_1(w;\phi), \ \Delta_2(w)=\Delta_2(w;\phi,g).
\end{align}
for notational simplicity.

\subsubsection{Part \ref{BB1}}
Suppose $g(0)<0$. 
From Assumptions \ref{AA2} and \ref{AA3} and the continuity of $\phi$ and $g$, 
for a sufficiently small $\epsilon>0$,
there exists $c_g>0$ and $w_0>0$ such that
\begin{equation}\label{case11111}
 g(w)\leq -c_g, \ 0\leq \phi(w)<\epsilon, \text{ and }  \phi'(w)\geq 0
\end{equation}
for $0<w<w_0$.
Clearly, by \eqref{case11111}, $\Delta_1(w)>0$ for $w\in(0,w_0)$. Further we consider the integral of 
$\Delta_2(t)/g(t)$ on $t\in (w,w_0)$.
By integration by parts, we have
\begin{align*}
 \int_w^{w_0}\frac{g'(t)}{g^2(t)}\{1+\phi(t)\}dt
&=\left[-\frac{1+\phi(t)}{g(t)}\right]_w^{w_0}+\int_w^{w_0}\frac{\phi'(t)}{g(t)}dt
 \leq \left[-\frac{1+\phi(t)}{g(t)}\right]_w^{w_0} \\
 &=-\frac{1+\phi(w_0)}{g(w_0)}+\frac{1+\phi(w)}{g(w)}\leq-\frac{1+\phi(w_0)}{g(w_0)}
\end{align*}
for $w\in (0,w_0)$, since $\phi'(t)/g(t)$ is nonpositive.
Hence we have
\begin{align*}
 \int_w^{w_0} \frac{\Delta_2(t)}{-g(t)}dt 
 &=\int_w^{w_0}
 \left(\frac{1}{t}-d_n\frac{g'(t)}{g(t)}-d_n \frac{g'(t)}{g^2(t)}\{1+\phi(t)\}\right)dt \\
 &\geq \log\frac{w_0}{w}-d_n\log\frac{|g(w_0)|}{|g(w)|}
 +d_n\frac{1+\phi(w_0)}{g(w_0)} \\
&\geq  \log\frac{w_0}{w}-d_n\log\frac{|g(w_0)|}{c_g}+d_n\frac{1+\phi(w_0)}{g(w_0)} 
\end{align*}
which goes to infinity as $w\to 0$.
Therefore $\Delta_2(w)$ (and hence $\Delta_1(w)+\Delta_2(w)$) takes positive value on $(0,w_0)$.
 Hence
\begin{align*}
 \Delta(w)= g(w) \left\{\Delta_1(w)+\Delta_2(w)\right\}
\end{align*}
takes negative value on $(0,w_0)$ since $g(w)<0$,
which contradicts $\Delta(w)\geq 0$ for any $w$.

\subsubsection{Part \ref{BB2}}
Suppose that there exists $w_1>0$ such that $g(w_1)<0$. 
Since $g(0)\geq 0$ by Part \ref{BB1} and $g(w)$ is continuous, 
there exists $w_2\in [0,w_1)$ such that
\begin{align}\label{eq:g_0}
g(w_2)=0, 
\ g(w)<0  \text{ for all }w_2< w\leq w_1.
\end{align}
Further Assumption \ref{AA2} ensures that 
there exists $w_3\in (w_2,w_1)$ such that $ \phi(w)$ is monotone on $(w_2,w_3)$.

Since $\phi(w)$ is bounded on $w\in (w_2,w_3)$, we have
\begin{equation}\label{BB2.1}
 \int_{w_2}^{w_3}\frac{\Delta_1(t)}{1+\phi(t)}dt=
2\int_{w_2}^{w_3}\frac{c_{p,n}-\phi(t)}{t\{1+\phi(t)\}}dt
+d_n\left[\log(1+\phi(t))\right]_{w_2}^{w_3},
\end{equation}
which is bounded from above and below when $w_2>0$ and goes to infinity when $w_2=0$.
Further since $g(w)<0$ for $w\in(w_2,w_3)$,
we have
\begin{align*}
 \frac{\Delta_2(w)}{1+\phi(w)} 
&= \frac{-g(w)}{w\{1+\phi(w)\}}+\frac{d_ng'(w)}{1+\phi(w)}+d_n\frac{g'(w)}{g(w)} \\
&\geq \frac{d_ng'(w)}{1+\phi(w)}+d_n\frac{g'(w)}{g(w)}.
\end{align*}
Then, by integration by parts, we have
\begin{align}
& \frac{1}{d_n}\int_{w_2}^{w_3}\left\{\frac{\Delta_2(t)}{1+\phi(t)}-d_n\frac{g'(t)}{g(t)}\right\}dt \notag\\
&\geq \int_{w_2}^{w_3}\frac{g'(t)}{1+\phi(t)}dt  =\left[\frac{g(t)}{1+\phi(t)}\right]_{w_2}^{w_3}
+\int_{w_2}^{w_3}\frac{g(t)\phi'(t)}{\{1+\phi(t)\}^2}dt \notag \\
&\geq   
\left[\frac{g(t)}{1+\phi(t)}\right]_{w_2}^{w_3}
-\max_{t\in[w_2,w_3]}|g(t)|\int_{w_2}^{w_3}\frac{|\phi'(t)|}{\{1+\phi(t)\}^2}dt \label{BB2.2}\\
&=
\left(\frac{g(w_3)}{1+\phi(w_3)}- \frac{g(w_2)}{1+\phi(w_2)}\right)
 -\max_{t\in[w_2,w_3]}|g(t)|\left|\frac{1}{1+\phi(w_2)}-\frac{1}{1+\phi(w_3)}\right|,\notag
\end{align}
which is bounded from below.
For $ w \in (w_2,w_3)$, we have
\begin{align*}
 \int_{w}^{w_3}\frac{g'(t)}{g(t)}dt=\log|g(w_3)|-\log|g(w)|
\end{align*}
which goes to $\infty$ as $w\to w_2$ since $g(w_2)=0$.
Then the integral
\begin{align*}
 \int_w^{w_3}\frac{\Delta_1(t)+\Delta_2(t)}{1+\phi(t)}dt
\end{align*}
goes to infinity as $w\to w_2$. 
Hence $\Delta_1(w)+\Delta_2(w)$ takes positive value on $(w_2,w_3)$ and
\begin{align*}
 \Delta(w)= g(w) \left\{\Delta_1(w)+\Delta_2(w)\right\}
\end{align*}
takes negative value on $(w_2,w_3)$ since $g(w)<0$, 
which contradicts $\Delta(w)\geq 0$ for any $w$.

\subsubsection{Part \ref{BB3}}
Suppose that there exists $w_1>w_0$ such that $g(w_1)=0$. 
 Assumption \ref{AA2} ensures that 
there exists $w_2\in (w_0,w_1)$ such that $ \phi(w)$ is monotone on $(w_2,w_1)$.
As in \eqref{BB2.1} and \eqref{BB2.2} of Part \ref{BB2},
 the integral
\begin{align*}
 \int_{w_2}^{w_1}\left\{\frac{\Delta_1(t)+\Delta_2(t)}{1+\phi(t)}-\frac{d_ng'(t)}{g(t)}\right\}dt 
\end{align*}
is bounded from above.
Further, for $ w \in (w_2,w_1)$, we have
\begin{align*}
 \int^{w}_{w_0}\frac{g'(t)}{g(t)}dt=\log g(w)-\log g(w_0)
\end{align*}
which goes to $-\infty$ as $w\to w_1$ since $g(w_1)=0$.
Hence $\Delta_1(w)+\Delta_2(w)$ takes negative value on $(w_0,w_1)$ and
\begin{align*}
 \Delta(w)= g(w) \left\{\Delta_1(w)+\Delta_2(w)\right\}
\end{align*}
takes negative value on $(w_0,w_1)$ since $g(w)>0$, 
which contradicts $\Delta(w)\geq 0$ for any $w$.

\subsection{Proof of Lemma \ref{CC0}}
\label{sec.ap.2}
 When $\theta=0$, the distribution of $W=\|X\|^2/S$ is $(p/n)F_{p,n}$ where
$F_{p,n}$ is a central $F$-distribution with $p$ and $n$ degrees of freedom.
Hence the tail behavior of the density
of $W$ is given by $f_W(w)\approx w^{-n/2-1}$. Therefore if $E[\phi(W)^2/W]<\infty$,
it must be that 
\begin{equation}
 \int_1^\infty \frac{\phi(t)^2}{t}t^{-n/2-1}dt=\int_1^\infty \frac{1}{t}\frac{\phi(t)^2}{t^{n/2+1}}dt<\infty.
\end{equation}
Since $ \int_1^\infty dt/t=\infty$, $\phi$ must satisfy
\begin{equation*}
\liminf_{t\to\infty}\frac{\phi(t)^2}{t^{n/2+1}}=0
\end{equation*}
which implies
\begin{align*}
\liminf_{t\to\infty} \left(\frac{|\phi(t)|^{4/(n+2)}}{t}\right)^{(n+2)/2}=
\liminf_{t\to\infty}\frac{|\phi(t)|^{d_n}}{t}=0.
\end{align*}

 \subsection{Proof of Lemma \ref{lem:CC}}
\label{sec.ap.3}
\subsubsection{Part \ref{CC1.1}}
 By \eqref{eq:CC1.lem}, it is clear that $ \phi_* \geq (p-2)/(n+2)$ and hence
 \begin{equation}\label{futoushiki}
  d_n(1+\phi_*)\geq 2\beta_\star>2b\beta_\star,
 \end{equation}
since $b<1$. Further \eqref{eq:CC1.1.lem} implies that
\begin{equation}\label{proof:assump}
 \Delta_1(w;\phi)- d_n\phi'(w)-\frac{2b\beta_\star}{w\log w}
=\frac{2}{w}\left(\frac{p-2}{n+2}-\phi(w)-\frac{b\beta_\star}{\log w}\right)
\leq 0, 
\end{equation}
for all $w\geq w_0$.

\subsubsection{Part \ref{CC1.2}}
\label{ap.subsub.CC1.2}
   Let $\alpha=2b\beta_\star$ and fix
\begin{equation}\label{epsilon_dayo}
 \epsilon =\frac{d_n(1+\phi_*)-\alpha}{6d_n}.
\end{equation}
 Then, by Assumption \ref{AA2} and $\lim_{w\to\infty}\phi(w)=\phi_*$,
 there exists $w_1$ such that
\begin{equation}\label{epseps}
\phi(w)\text{ is monotone}, \text{and } \int_w^\infty  \left|\phi'(t)\right|dt=\left|\phi_*-\phi(w)\right|<\epsilon
\end{equation}
for all $w\geq w_1$.
 Since $g(w)\not\equiv 0$ and $g(w)$ satisfies
 \ref{BB1}, \ref{BB2} and  \ref{BB3} of Lemma \ref{lem:GG}, there exists $w_2>0$ such that
 $g(w)>0$ for all $w\geq w_2$. 
Define $w_3=\max(w_0,w_1,w_2,1)$ and consider the integral 
 \begin{align*}
 \int_{w_3}^w \frac{\Delta_2(t;\phi,g)+d_n\phi'(t)+\alpha/(t\log t)}{g(t)} dt\leq
 \sum_{i=1}^4h_i(w;w_3)
 \end{align*}
where
\begin{equation}\label{h1____4}
\begin{split}
 h_1(w;w_3) &=\int_{w_3}^w \left(-\frac{1}{t}+d_n\frac{g'(t)}{g(t)}\right)dt, \\
 h_2(w;w_3) &=d_n\int_{w_3}^w \left(\frac{g'(t)}{g^2(t)}\{1+\phi(t)\}-\frac{\phi'(t)}{g(t)}\right)dt, \\ 
h_3(w;w_3)&=2d_n\int_{w_3}^w\frac{|\phi'(t)|dt}{g(t)}, \\  
h_4(w;w_3)&= \alpha\int_{w_3}^w\frac{1}{g(t)t \log t}dt.
\end{split} 
\end{equation}
 We are going to show
\begin{align*}
 \liminf_{w\to\infty}\sum\nolimits_{i=1}^4h_i(w;w_3)=-\infty
\end{align*}
which guarantees that there exists $w_*\in (w_3,\infty)$ such that
\begin{align*}
 \Delta_2(w_*;\phi,g)+d_n\phi'(w_*)+ \frac{\alpha}{w_*\log w_*}<0.
\end{align*}
The first term is 
 \begin{equation}\label{eq:h1_1}
\begin{split}
h_1(w;w_3)& =\int_{w_3}^w \left(-\frac{1}{t}+d_n\frac{g'(t)}{g(t)}\right)dt \\
  &=-\log\frac{w}{w_3}+d_n\log\frac{g(w)}{g(w_3)} \\
  &=\log \frac{g(w)^{d_n}}{w}+ \log \frac{w_3}{g(w_3)^{d_n}}.
\end{split}
 \end{equation}
Since $g\in \mathcal{G}$, $\liminf_{w\to\infty}g(w)^{d_n}/w=0$ by Lemma \ref{CC0}.
Hence we have
\begin{equation}\label{liminf_1}
 \liminf_{w\to\infty}h_1(w;w_3)=-\infty.
\end{equation}
By integration by parts, the second term, $h_2(w;w_3)$, divided by $d_n$ is
\begin{equation}\label{int_by_parts_h_2}
 \begin{split}
 \frac{h_2(w;w_3)}{d_n}&
 =\int_{w_3}^w \left(\frac{g'(t)}{g^2(t)}\{1+\phi(t)\}-\frac{\phi'(t)}{g(t)}\right)dt \\
 &=\left[-\frac{1+\phi(t)}{g(t)}\right]_{w_3}^w \\
 &=-\frac{1+\phi(w)}{g(w)}+\frac{1+\phi(w_3)}{g(w_3)} \\
& \leq -\frac{1+\phi_*-\epsilon}{g(w)}+\frac{1+\phi(w_3)}{g(w_3)}. 
\end{split}
\end{equation}
 Let
\begin{equation}\label{gG}
 G(w)=\frac{1}{g(w)\log w}
\end{equation}
 and recall $w_3$ is greater than $1$.
 Then, with \eqref{gG},
 $h_3(w;w_3)$ and $h_4(w;w_3)$ for $w> w_3>1$, are bounded as follows:
\begin{equation}\label{h333}
 \begin{split}
  h_3(w;w_3) 
  &=2d_n\int_{w_3}^w G(t)\log t |\phi'(t)|dt
  \\& \leq 2d_n\log w\sup_{t\in(w_3,w)} G(t) 
\int_{w_3}^w |\phi'(t)|dt \\ &<2d_n\epsilon \log w \sup_{t\in(w_3,w)}G(t), 
 \end{split}
\end{equation} 
by \eqref{epseps}, and 
\begin{equation}\label{h444}
 \begin{split}
  h_4(w;w_3)
  &=\alpha\int_{w_3}^w \frac{G(t)dt}{t } \\
  &\leq \alpha\sup_{t\in(w_3,w)}G(t)\int_{w_3}^w\frac{dt}{t}\\
  &\leq \alpha \log w\sup_{t\in(w_3,w)}G(t). 
\end{split}
\end{equation}
Thus, by \eqref{int_by_parts_h_2}, \eqref{h333} and \eqref{h444}, we have
\begin{equation}\label{eq:h2_4}
\begin{split}
& \sum\nolimits_{i=2}^4 h_i(w;w_3)-\frac{1+\phi(w_3)}{g(w_3)}\\&\leq  
  \log w\left\{\right(\alpha+2d_n\epsilon) \sup_{t\in(w_3,w)}G(t) - d_n(1+\phi_*-\epsilon) G(w)\}. 
 \end{split} 
\end{equation}
{\bfseries  Case I}: $\limsup_{w\to\infty} G(w)=\infty$

Since there exists $w_4>w_3$ such that $  \sup_{t\in(w_3,w_4)}G(t)=G(w_4)>1$,
we have
\begin{align*}
 (\alpha+2d_n\epsilon) \sup_{t\in(w_3,w_4)}G(t)   - d_n(1+\phi_*-\epsilon) G(w_4)
=-G(w_4)\frac{d_n(1+\phi_*)-\alpha}{2}.
\end{align*}
Therefore, by \eqref{eq:h2_4},
 \begin{equation}\label{case.1.h234}
 \sum\nolimits_{i=2}^4 h_i(w_4;w_3)-\frac{1+\phi(w_3)}{g(w_3)}
  \leq
-G(w_4)\log w_4\frac{d_n(1+\phi_*)-\alpha}{2}.
 \end{equation}
By \eqref{eq:h1_1} and \eqref{gG}, we have
\begin{equation}\label{h1111}
 \begin{split}
h_1(w_4;w_3) - \log\frac{w_3}{g(w_3)^{d_n}}  
  &=\log\frac{g(w_4)^{d_n}}{w_4}  \\
  &=\log \frac{1}{w_4\{G(w_4)\}^{d_n}(\log w_4)^{d_n}} \\
  & =-d_n\log\log w_4-\log w_4 - d_n \log G(w_4) \\
  & \leq -d_n\log\log w_4-\log w_4 ,
 \end{split}
\end{equation}
 since $G(w_4)>1$. 
 By \eqref{case.1.h234}, \eqref{h1111} and choosing $w_4$ to be sufficiently large,
 we conclude that
\begin{equation}\label{liminf_2}
 \liminf_{w\to\infty}\sum\nolimits_{i=1}^4 h_i(w;w_3)=-\infty.
\end{equation}

 \medskip
 
 {\bfseries Case II}: $\limsup_{w\to\infty} G(w)=G_*\in (0,\infty)$

Under the choice of $\epsilon$ given by \eqref{epsilon_dayo}, fix 
 \begin{equation}\label{nunu}
\nu=\frac{G_*\{d_n(1+\phi_*)-\alpha\}}{4\{\alpha+d_n(1+\phi_*+\epsilon)\}}.
 \end{equation}
There exists $w_5\geq w_3$ such that $ \sup_{t\geq w_5}G(t)<G_*+\nu$ and 
$w_6\in(w_5,\infty)$ which satisfies $G(w_6)\geq G_*-\nu$ can be taken.
Then we have 
\begin{equation}\label{case222}
 \begin{split}
&(\alpha+2d_n\epsilon) \sup_{t\in(w_5,w_6)}G(t)   - d_n(1+\phi_*-\epsilon) G(w_6)  \\
 &\leq (\alpha+2d_n\epsilon)(G_*+\nu) - d_n(1+\phi_*-\epsilon)(G_*-\nu) \\
  &=\nu\{(\alpha+2d_n\epsilon)+d_n(1+\phi_*-\epsilon)\} \\ &\quad
  +G_*\{ (\alpha+2d_n\epsilon)-d_n(1+\phi_*-\epsilon)\} \\
 &=\nu\{\alpha+d_n(1+\phi_*+\epsilon)\}- G_*( \{d_n(1+\phi_*)-\alpha\}-3d_n\epsilon) \\
 &=G_*\frac{d_n(1+\phi_*)-\alpha}{4}
- G_*\left( d_n(1+\phi_*)-\alpha-\frac{d_n(1+\phi_*)-\alpha}{2}\right) \\
 &=-G_*\frac{d_n(1+\phi_*)-\alpha}{4}
\end{split} 
\end{equation}
by \eqref{epseps} and \eqref{nunu}.
 Hence, by \eqref{eq:h2_4} and \eqref{case222}, we have
 \begin{equation}\label{caseII.II.1}
  \sum\nolimits_{i=2}^4 h_i(w_6;w_5)- \frac{1+\phi(w_5)}{g(w_5)}
\leq -G_*\frac{d_n(1+\phi_*)-\alpha}{4}\log w_6.
 \end{equation}
As in \eqref{h1111}, we have
\begin{equation}\label{caseII.II.2}
 \begin{split}
&h_1(w_6;w_5) - \log \frac{w_5}{g(w_5)^{d_n}} \\
    & =-d_n\log\log w_6-\log w_6 - d_n \log G(w_6) \\
  & \leq -d_n\log\log w_6-\log w_6 -d_n \log(G_*-\nu).
 \end{split}
\end{equation} 
 By choosing $w_6$ to be sufficiently large on
 \eqref{caseII.II.1} and \eqref{caseII.II.2}, we have
\begin{equation}\label{liminf_4}
  \liminf_{w\to\infty}\sum\nolimits_{i=1}^4 h_i(w;w_5)=-\infty.
\end{equation}
 
 \medskip
 
{\bfseries Case III}: $\limsup_{w\to\infty} G(w)=0$ or equivalently $\lim_{w\to\infty} G(w)=0$

 {\bfseries  Case III-i}: $\limsup_{w\to\infty} G(w)w^{1/(4d_n)}<\infty$

Let $\tau=1/(4d_n)>0$. Note 
 \begin{align*}
 h_3(w;w_3) &=2d_n\int_{w_3}^w G(t)\log t |\phi'(t)|dt \\
  &\leq 2d_n \int_{w_3}^\infty \{G(t)t^{\tau}\}\frac{\log t }{t^\tau}|\phi'(t)|dt \\
  &\leq 2d_n \sup_{t\in(w_3,\infty)}G(t)t^{\tau}\sup_{t\in(w_3,\infty)}\frac{\log t }{t^\tau}
  \int_{w_3}^\infty |\phi'(t)|dt \\
  &\leq 2d_n\epsilon \sup_{t\in(w_3,\infty)}G(t)t^{\tau}\sup_{t\in(w_3,\infty)}\frac{\log t }{t^\tau},
 \end{align*}
 which is bounded from above.
 Also note
 \begin{align*}
  h_4(w;w_3)
  &=\alpha\int_{w_3}^w \frac{G(t)dt}{t } \\
  &\leq \alpha\int_{w_3}^\infty \frac{G(t)t^\tau dt}{t^{1+\tau}}\\
  & \leq \alpha\sup_{t\in(w_3,\infty)}G(t)t^{\tau}
  \int_{w_3}^\infty \frac{dt}{t^{1+\tau}}
 \end{align*} 
which is bounded from above. 
Further we have $\liminf_{w\to\infty} h_1(w;w_3)=-\infty $ by \eqref{liminf_1}
 and $h_2(w;w_3)\leq \{1+\phi(w_3)\}/g(w_3)$ by \eqref{int_by_parts_h_2}.
 Therefore
we have
\begin{equation}\label{case.3.2.liminf.final}
  \liminf_{w\to\infty}\sum\nolimits_{i=1}^4 h_i(w;w_3)=-\infty.
\end{equation} 

 \medskip

 {\bfseries  Case III-ii}: $\limsup_{w\to\infty} G(w)w^{1/(4d_n)}=\infty$

 Under the choice of $\epsilon$ given by \eqref{epsilon_dayo}, there exists $w_7\geq w_3$  such that
\begin{equation}\label{alphadneps}
\sup_{t\in(w_7,\infty)}G(t)<\frac{1}{2(\alpha+2d_n\epsilon)}.
\end{equation} 
By \eqref{alphadneps}, we have
 \begin{align*}
  \sum\nolimits_{i=2}^4 h_i(w;w_7)-\frac{1+\phi(w_7)}{g(w_7)}&\leq (\alpha+2d_n\epsilon)\sup_{t\in(w_7,w)}G(t)\log w \\
  &\leq \frac{\log w}{2},
 \end{align*}
 \begin{align*}
-\frac{3}{4}\log w+  \sum\nolimits_{i=2}^4 h_i(w;w_7)\leq 
-\frac{\log w}{4}  +\frac{1+\phi(w_7)}{g(w_7)}
 \end{align*}
and hence 
 \begin{equation}\label{case.3.1.lim}
\lim_{w\to\infty} \left\{-\frac{3}{4}\log w +\sum\nolimits_{i=2}^4 h_i(w;w_7)\right\}=-\infty.
 \end{equation}
Recall $G(w)=1/\{g(w)\log w\}$. Then we have
\begin{align*}
 & h_1(w;w_7)+\log\frac{g(w_7)^{d_n}}{w_7}+\frac{3}{4}\log w \\
 &=\log\frac{g(w)^{d_n}}{w} +\frac{3}{4}\log w\\
 & = -d_n\log\left\{G(w)w^{1/(4d_n)} \right\}-d_n\log\log w.
\end{align*}
Since $\limsup_{w\to\infty} G(w)w^{1/(4d_n)}=\infty$,
 \begin{equation}\label{case.3.1.liminf}
 \liminf_{w\to\infty} \left\{h_1(w;w_7)+\frac{3}{4}\log w\right\}=-\infty
 \end{equation}
follows. Note 
\begin{equation}\label{case.3.1.bunkai}
 \begin{split}
 &  \sum\nolimits_{i=1}^4 h_i(w;w_7) \\
 &=
 \left\{ h_1(w;w_7)+\frac{3}{4}\log w\right\} 
 +
\left\{\sum_{i=2}^4 h_i(w;w_7)-\frac{3}{4}\log w\right\}.
\end{split}
\end{equation}
 By \eqref{case.3.1.lim}, \eqref{case.3.1.liminf} and \eqref{case.3.1.bunkai},
we have
\begin{equation}\label{case.3.1.liminf.final}
  \liminf_{w\to\infty}\sum\nolimits_{i=1}^4 h_i(w;w_7)=-\infty.
\end{equation} 
 
 \subsubsection{Part \ref{CC2.1}}
We have 
\begin{align*}
 \frac{w}{\phi(w)}\left(\Delta_1(w;\phi)+d_n\phi'(w)\right) 
=2\left(\frac{c_{p,n}}{\phi(w)}-1+d_n \frac{w\phi'(w)}{\phi(w)}\right). 
\end{align*}
By Assumption \ref{AA4} and the assumption $n\geq 3$ as in \eqref{eq:dim},
we have
\begin{equation*}
 d_n\limsup_{w\to\infty}w\frac{\phi'(w)}{\phi(w)} \leq d_n=\frac{4}{n+2} <1.
\end{equation*}
Since $\lim_{w\to\infty}1/\phi(w)=0$, there exists $w_1$ such that
\begin{align}
 \Delta_1(w;\phi)+d_n\phi'(w)\leq 0
\end{align}
 for all $w\geq w_1$.
 
\subsubsection{Part \ref{CC2.2}}
Consider the integral 
\begin{align*}
 \int_{w_1}^w \frac{\Delta_2(t;\phi,g)-d_n\phi'(t)}{g(t)} dt=
 h_1(w;w_1)+h_2(w;w_1)
\end{align*}
where $h_1(w;\cdot)$ and $h_2(w;\cdot)$ are given by \eqref{h1____4}.
 We are going to show
\begin{equation}\label{liminf_infty}
 \liminf_{w\to\infty}\{h_1(w;w_1)+h_2(w;w_1)\}=-\infty
\end{equation}
which guarantees that there exists $w_*\in (w_1,\infty)$ such that
\begin{align*}
 \Delta_2(w_*;\phi,g)-d_n\phi'(w_*)<0.
\end{align*}
By \eqref{liminf_1}, $ \liminf_{w\to\infty}h_1(w;w_1)=-\infty$ follows.
Also, by \eqref{int_by_parts_h_2}, $h_2(w;w_1)\leq \{1+\phi(w_1)\}/g(w_1)$.
Therefore \eqref{liminf_infty} follows.

\subsubsection{Part \ref{CC3.1}}
 By \eqref{eq:CC3.lem}, we have $\phi_*\leq (p-2)/(n+2)=c_{p,n}$ and hence
 \begin{equation}\label{futoushiki_2}
  d_n(1+\phi_*)\leq 2\beta_\star<2b\beta_\star
 \end{equation}
since $b>1$. When $\phi_*=c_{p,n}$,
$\phi(w)$ is ultimately monotone nondecreasing
 and hence without the loss of generality, $\phi'(w)\geq 0$ for all $w\geq w_0$.
 Then we have
\begin{equation}\label{main.part.2.delta.1.1}
 \begin{split}
  \Delta_1(w;\phi)-\frac{2b\beta_\star}{w\log w} 
& = 2\frac{c_{p,n}-\phi(w)}{w}+\phi'(w)-\frac{2b\beta_\star}{w\log w} \\
& \geq \frac{2}{w}\left( c_{p,n}-\frac{b\beta_\star}{\log w}-\phi(w)\right)\geq   0, 
 \end{split}
\end{equation} 
 for all $w\geq w_0$ by \eqref{eq:CC3.lem}.
 
Consider the case where $c_{p,n}-\phi_*=\delta>0$. 
 By Assumption \ref{AA4}, there exists $w_2$ such that
\begin{equation*}
 w\frac{\phi'(w)}{\phi(w)}> -\frac{\delta}{4\phi_*}
\end{equation*}
for all $w\geq w_2$.
 Further, 
 by $\lim_{w\to\infty} \phi(w)=\phi_*$, there exists $w_3$ such that
\begin{equation*}
 |\phi(w)-\phi_*|<\frac{\delta}{4\left\{1+\delta/(4\phi_*)\right\}}
\end{equation*}
for all $w\geq w_3$. Then,  for all $w\geq \max(w_2,w_3,e^{4b\beta_\star/\delta})$, we have
\begin{equation}\label{main.part.2.delta.1.2}
 \begin{split}
&\frac{w}{2}\left(\Delta_1(w;\phi)-\frac{2b\beta_\star}{w\log w}\right)  \\
  &  = c_{p,n}-\phi(w)+\phi(w)\frac{w\phi'(w)}{\phi(w)}-\frac{b\beta_\star}{\log w} \\
  &\geq\delta-\frac{\delta}{4\left\{1+\delta/(4\phi_*)\right\}} 
   -\left(\phi_*+\frac{\delta}{4\left\{1+\delta/(4\phi_*)\right\}} \right)
  \frac{\delta}{4\phi_*}-\frac{\delta}{4} \\
&=\frac{\delta}{4}.
 \end{split}
\end{equation} 
 Hence, under the condition \eqref{futoushiki_2},
 by \eqref{main.part.2.delta.1.1} and \eqref{main.part.2.delta.1.2},
 there exists $w_1$ such that
\begin{equation}
 \Delta_1(w;\phi)-\frac{2\beta}{w\log w}\geq 0
\end{equation}
for all $w\geq w_1$.

\subsubsection{Part \ref{CC3.2}}
There exists $w_4$ such that $\phi_*-\nu<\phi(w)<\phi_*+\nu$ for all $ w\geq w_4$.
Recall
\begin{align*}
 \Delta_2(w;\phi,g)=\frac{-g(w)}{w}+d_ng'(w)+d_n \frac{g'(w)}{g(w)}\{1+\phi(w)\}.
\end{align*} 
Hence, for all $w\geq \max(e,w_4,w_\sharp)$, we have
\begin{align*}
 & \Delta_2(w;\phi,\{\log (w+e)\}^{-1-\nu}k(w))\\
 &=\frac{-k(w)}{w\{\log (w+e)\}^{1+\nu}}
-\frac{d_n(1+\nu)k(w)}{(w+e)\{\log (w+e)\}^{2+\nu}}+\frac{d_n k'(w)}{\{\log (w+e)\}^{1+\nu}} \\
 &\quad +d_n\left(\frac{k'(w)}{k(w)}-\frac{1+\nu}{(w+e)\log (w+e)}\right)\{1+\phi(w)\} \\
 &\geq -\frac{d_n(1+\nu)(1+\phi_*+\nu)}{w\log w}
 -\frac{d_n(1+\nu)+1}{w\{\log w\}^{1+\nu}} \\
 &\geq -\frac{d_n(1+\phi_*)}{w\log w}-\frac{\nu d_n(3+\phi_*)}{w\log w} 
 -\frac{2d_n+1}{w\{\log w\}^{1+\nu}} \\
 &= -\frac{\alpha }{w\log w}
 +\frac{\{\alpha -d_n(1+\phi_*)-\nu d_n(3+\phi_*)\}(\log w)^\nu-(2d_n+1)}{w\{\log w\}^{1+\nu}} \\
 &\geq -\frac{\alpha }{w\log w}
 +\frac{\{\alpha -d_n(1+\phi_*)\}(\log w)^\nu-2(2d_n+1)}{2w\{\log w\}^{1+\nu}}.
\end{align*} 
 Let $ w_*=\max(e,w_4,w_5)$ where
 \begin{equation*}
w_5=  \exp\left\{\left(\frac{2(2d_n+1)}{\alpha -d_n(1+\phi_*)}\right)^{1/\nu}\right\}.
 \end{equation*}
Then, for all $w\geq \max(w_*, w_\sharp)$, we have
\begin{equation*}
 \Delta_2(w;\phi,\{\log (w+e)\}^{-1-\nu}k(w))\geq 
-\frac{\alpha }{w\log w}.
\end{equation*}

\subsection{Boundedness of $\|z\|\|\nabla \log m(\|z\|;a,L)\|$}
\label{sec:ap.bound}
Note 
\begin{equation*}
 \nabla \log m(\|z\|;a,L)=-z
  \frac{\int_0^1\lambda^{p/2+a+1}L(1/\lambda)\exp(-\|z\|^2\lambda/2)d\lambda}
  {\int_0^1\lambda^{p/2+a}L(1/\lambda)\exp(-\|z\|^2\lambda/2)d\lambda}.
\end{equation*}
We have $\|z\|\|\nabla \log m(\|z\|;a,L)\|=0$ at $\|z\|=0$.
 By a Tauberian Theorem which is also applied in \eqref{TaubTaub},
\begin{align*}
 \lim_{\|z\|\to\infty}\|z\|M(\|z\|;a,L)
  &=\lim_{\|z\|\to\infty}\|z\|^2\frac{\int_0^1\lambda^{p/2+a+1}L(1/\lambda)\exp(-\|z\|^2\lambda/2)d\lambda}
 {\int_0^1\lambda^{p/2+a}L(1/\lambda)\exp(-\|z\|^2\lambda/2)d\lambda} \\
 &=p+2a+2,
\end{align*}
the boundedness of $\|\nabla \log m(\|z\|;a,L)\|$ follows.

  \subsection{Derivation of \eqref{known_variance_phi_asymptotic} and \eqref{lim.final.qa}}
\label{sec.ap.5}
\subsubsection{Derivation of \eqref{known_variance_phi_asymptotic}}
Recall
\begin{equation*}
 \psi_{-2,b}(v)=v\frac{\int_0^1\lambda^{p/2-1}\{\log(1/\lambda)\}^b\exp(-v\lambda/2)d\lambda}
  {\int_0^1\lambda^{p/2-2}\{\log(1/\lambda)\}^b\exp(-v\lambda/2)d\lambda}.
\end{equation*}
By integration by parts,
\begin{align*}
& \frac{v}{2}\int_0^1\lambda^{p/2-1}\{\log(1/\lambda)\}^b\exp(-v\lambda/2)d\lambda \\
 & =\left[-\lambda^{p/2-1}\{\log(1/\lambda)\}^b\exp(-v\lambda/2)\right]_0^1 \\
 & \quad +(p/2-1) \int_0^1\lambda^{p/2-2}\{\log(1/\lambda)\}^b\exp(-v\lambda/2)d\lambda \\
& \quad -b\int_0^1\lambda^{p/2-1}\{\log(1/\lambda)\}^{b-1}\lambda^{-1}\exp(-v\lambda/2)d\lambda.
\end{align*}
Thus we have
\begin{equation*}
 \psi_{-2,b}(v)=p-2-2b\frac{\int_0^1\lambda^{p/2-2}\{\log(1/\lambda)\}^{b-1}\exp(-v\lambda/2)d\lambda}
  {\int_0^1\lambda^{p/2-2}\{\log(1/\lambda)\}^b\exp(-v\lambda/2)d\lambda}.
\end{equation*}
By a Tauberian theorem as in \eqref{TaubTaub}, we have 
\begin{align*}
 \lim_{v\to\infty}
  (\log v)\frac{\int_0^1\lambda^{p/2-2}\{\log(1/\lambda)\}^{b-1}\exp(-v\lambda/2)d\lambda}
  {\int_0^1\lambda^{p/2-2}\{\log(1/\lambda)\}^b\exp(-v\lambda/2)d\lambda} =1,
\end{align*}
and hence 
\begin{equation}
 \lim_{v\to\infty}\left(\log v\right)\left\{p-2-\psi_{-2,b}(v)\right\}=2b.
\end{equation}
\subsubsection{Derivation of \eqref{lim.final.qa}}
Recall 
\begin{equation}\label{saigo.1}
 \phi_{-2,b}(w)
  =w\frac{\int_0^1\lambda^{p/2-1}\{\log(1/\lambda)\}^b(1+w\lambda)^{-(p+n)/2-1}d\lambda}
  {\int_0^1\lambda^{p/2-2}\{\log(1/\lambda)\}^b(1+w\lambda)^{-(p+n)/2-1}d\lambda}.
\end{equation}
Note
\begin{align*}
 (1+w\lambda)^{-(p+n)/2-1}=(1+w\lambda)^{-p/2+1}(1+w\lambda)^{-n/2-2},
\end{align*}
\begin{equation*}
\frac{d}{d\lambda}\left\{\frac{(1+w\lambda)^{-n/2-1}}{w(-n/2-1)}\right\} =(1+w\lambda)^{-n/2-2},
\end{equation*}
and
\begin{equation*}
\frac{d}{d\lambda} \left(\frac{\lambda}{1+w\lambda}\right)^{p/2-1}=
  (p/2-1)\left(\frac{\lambda}{1+w\lambda}\right)^{p/2-2}\frac{1}{(1+w\lambda)^2}.
\end{equation*}
Then, by integration by parts, we have
\begin{equation}\label{saigo.2}
 \begin{split}
& (n/2+1)w
 \int_0^1\lambda^{p/2-1}\left\{\log\frac{1}{\lambda}\right\}^{b}(1+w\lambda)^{-(p+n)/2-1}d\lambda \\
 &=\left[(1+w\lambda)^{-n/2-1}\left(\frac{\lambda}{1+w\lambda}\right)^{p/2-1}
 \left\{\log\frac{1}{\lambda}\right\}^{b}\right]^1_0 \\
&\quad + (p/2-1)\int_0^1\left(\frac{\lambda}{1+w\lambda}\right)^{p/2-2}\frac{(1+w\lambda)^{-n/2-1}}{(1+w\lambda)^2}\left\{\log\frac{1}{\lambda}\right\}^{b}
 d\lambda \\
 &\quad -b
 \int_0^1\left(\frac{\lambda}{1+w\lambda}\right)^{p/2-1}
 \frac{(1+w\lambda)^{-n/2-1}}{(1+w\lambda)^2}
 \left\{\log\frac{1}{\lambda}\right\}^{b-1}\frac{1}{\lambda} d\lambda,
\end{split}
\end{equation}
which is equal to
\begin{equation}\label{saigo.3}
 \begin{split}
 &(p/2-1)\int_0^1\lambda^{p/2-2}\left\{\log\frac{1}{\lambda}\right\}^b(1+w\lambda)^{-(p+n)/2-1}
   d\lambda \\
 &\quad -b\int_0^1\lambda^{p/2-2} \left\{\log\frac{1}{\lambda}\right\}^{b-1}
(1+w\lambda)^{-(p+n)/2} d\lambda.
\end{split}
\end{equation}
By \eqref{saigo.1}, \eqref{saigo.2} and \eqref{saigo.3}, we have
\begin{align}\label{saigo.4}
 \phi_{-2,b}=\frac{p-2}{n+2} -\frac{2b}{n+2}
\frac{\int_0^1\lambda^{p/2-2} \left\{\log(1/\lambda)\right\}^{b-1}
 (1+w\lambda)^{-(p+n)/2} d\lambda}
 {\int_0^1\lambda^{p/2-2}\{\log(1/\lambda)\}^b(1+w\lambda)^{-(p+n)/2-1}d\lambda}
\end{align}
As in \eqref{phi.L.limit}, we have
\begin{equation}\label{saigo.5}
 \begin{split} 
& \lim_{w\to\infty}(\log w)
\frac{\int_0^1\lambda^{p/2-2} \left\{\log(1/\lambda)\right\}^{b-1}
 (1+w\lambda)^{-(p+n)/2} d\lambda}
 {\int_0^1\lambda^{p/2-2}\{\log(1/\lambda)\}^b(1+w\lambda)^{-(p+n)/2-1}d\lambda} \\
 &= \frac{\int_0^\infty t^{p/2-2} (1+t)^{-(p+n)/2} dt}
 {\int_0^\infty t^{p/2-2}(1+t)^{-(p+n)/2-1}dt} \\
&=\frac{p+n}{n+2}
\end{split}
\end{equation}
and hence, by \eqref{saigo.4} and \eqref{saigo.5}, \eqref{lim.final.qa} follows.

\end{document}